\documentclass[12pt]{article}
\usepackage{mathtools}
\usepackage{graphicx}
\usepackage{amsmath,amsfonts,latexsym,amscd,amssymb}
\usepackage{graphicx}
\usepackage[ansinew]{inputenc}
\usepackage{epsfig}
\usepackage{amsmath,amsthm}
\usepackage{amssymb}
\usepackage{afterpage}
\usepackage{selinput}
\usepackage{fancyhdr}
\def\A{{\cal A}}

\newcommand{\ba}{\begin{eqnarray}}
\newcommand{\ea}{\end{eqnarray}}

\newtheorem{thm}{Theorem}[section]

\newtheorem{conjecture}{Conjecture}

\newtheorem{theorem}[thm]{Theorem}

\newtheorem{lemma}[thm]{Lemma}

\newtheorem{claim}[thm]{Claim}

\makeatletter
\newcommand*{\rom}[1]{\expandafter\@slowromancap\romannumeral #1@}
\makeatother

\date{}

\begin{document}
\author{Batoul Tarhini$^{1,2}$, Olivier Togni$^{1}$} \footnotetext[1]{LIB Laboratory, University of Burgundy, Dijon, France.}
 \footnotetext[2]{KALMA Laboratory, Faculty of Sciences, Lebanese University, Beirut, Lebanon.}
     
\title{$S$-Packing Coloring of Cubic Halin Graphs}
\maketitle
\bibliographystyle{alpha}

\begin{abstract}
Given a non-decreasing sequence $S = (s_{1}, s_{2}, \ldots , s_{k})$ of positive integers, an $S$-packing coloring of a graph $G$ is a partition of the vertex set of $G$ into $k$ subsets $\{V_{1}, V_{2}, \ldots , V_{k}\}$ such that for each $1 \leq i \leq k$, the distance between any two distinct vertices $u$ and $v$ in $V_{i}$ is at least $s_{i} + 1$. In this paper, we study the problem of $S$-packing coloring of cubic Halin graphs, and we prove that every cubic Halin graph is $(1,1,2,3)$-packing colorable. In addition, we prove that such graphs are $(1,2,2,2,2,2)$-packing colorable.
\end{abstract}

\section{Introduction}
\pagenumbering{arabic}
For a sequence of non-decreasing positive integers $S = (s_{1}, \dots , s_{k})$, an {\em $ S$-packing coloring} of a graph $G$ is a partition of $V (G)$ into sets $V_{1}, \dots , V_{k} $ such that for each $1 \leq i \leq k$ and $u\ne v \in V_{i}$, we have $ d(u,v)\geq s_{i}+1$, where $d(u,v)$ is the distance between $u$ and $v$ in $G$.\\ The smallest $k$ such that $G$ has a $(1, 2,\dots , k)$-packing coloring ($k$-packing coloring)
is called the {\em packing chromatic number} of $G$ and is denoted by $\chi_{p}(G)$.\\ Graphs considered in this paper are simple, having no loops or multiple edges. For a graph $G$, we denote by $V(G)$ the set of vertices of $G$ and $E(G)$ the set of its edges.  
We denote by $S(G)$ the graph obtained from $G$ by subdividing every edge.\\A subcubic graph is a graph whose maximum degree is at most 3. A cubic graph $G$ is a graph in which every vertex has degree 3.\\ For a vertex $v$ in a graph $G$, we denote by $N(v)$ the set of neighbors of $v$ in $G$.\\Let $C=a_{1} a_{2} \dots a_{n}$ be a cycle in $G$. Denote by $C_{[a_{i},a_{j}]}$ the part of the cycle which is the path $a_{i} a_{i+1}\dots a_{j}$, and by $l(C_{[a_{i},a_{j}]})$ the length of this part of the cycle where $i+1=1$ for $i=n$. We say that two vertices $u$ and $v$ are {\em consecutive} on $C$ if $uv\in E(C)$.\\ A proper $k$-coloring of a graph $G$ is a mapping $\phi: V(G) \longrightarrow \{1, 2,\dots ,k\}$ that assigns to each vertex $v$ of $G$ a color $i \in [k]$ such that for any two adjacent vertices $u$ and $v$ of $G$, we have $\phi(u) \neq \phi (v) $. In this case, we say that $G$ is \textit{$k$-colorable}. It is clear that $(1,1,\dots ,1)$-packing colorings are the standard proper colorings of a graph. %that the proper coloring is a special case of the $S$-packing coloring when the sequence $S$ is equal to $(1,1,\dots ,1) $.
\\A \textit{bipartite graph} is a graph whose vertices can be divided into two disjoint stable sets $U$ and $V$ such that every edge connects a vertex in $U$ to a vertex in $V$. It is clear that any tree is a bipartite graph and that any bipartite graph is 2-colorable.\\A Halin graph is a planar graph constructed by connecting the leaves of a tree into a cycle such that the tree is of order at least four, and in which the degree of each vertex is either one which is called a leaf, or at least 3. For a Halin graph $G$, we write $G=T\cup C$ where $T$ is its characteristic tree and $C$ its adjoint cycle.\\
%it should be drawn in the plane so none of its edges cross (this is called planar embedding),
 %And the cycle connects the leaves in their clockwise ordering in this embedding. Thus, the cycle forms the outer face of the Halin graph, with the tree inside it.[1]
The questions on finding a bound of $\chi_{p}(G)$ and $\chi_{p}(S(G))$ if $G$ is a subcubic graph were discussed in several papers \cite{8,9,16,24,25}.
In particular, Gastineau and Togni \cite{16} asked whether $\chi_{p}(S(G)) \leq 5$ for every subcubic graph $G$ and Bre\v{s}ar, Klav\v{z}ar, Rall, and Wash \cite{9} later conjectured this.

\begin{conjecture}[Bre\v{s}ar, Klav\v{z}ar, Rall, and Wash \cite{9}] Let $G$ be a subcubic graph, then $\chi_{p}(S(G)) \leq 5$. \end{conjecture}

Due to the fact that if $G$ is $(s_{1},s_{2},...,s_{k})$-packing colorable then $S(G)$ is $(1,2s_{1}+1,2s_{2}+1,...,2s_{k}+1)$-packing colorable, Gastineau and Togni \cite{16} observed that if a graph $G$ is $(1, 1, 2, 2)$-packing colorable then $\chi_{p}(S(G)) \leq 5$. To prove that a subclass of subcubic graphs satisfies Conjecture 1, by the previous observation, it is sufficient to prove that it is $(1,1,2,2)$-packing colorable. They also asked the stronger question of whether every subcubic graph except the Petersen graph is $(1, 1, 2, 3)$-packing colorable.
Gastineau and Togni \cite{16} showed that subcubic graphs
are $(1, 2, 2, 2, 2, 2, 2)$-packing colorable and $(1, 1, 2, 2, 2)$-packing colorable. Balogh, Kostochka and Liu \cite{3} proved that subcubic graphs are $(1, 1, 2, 2, 3, 3, 4)$-packing colorable with color 4 used at most once and 2-degenerate subcubic graphs are $(1, 1, 2, 2, 3, 3)$-packing colorable. Moreover, Borodin and Ivanova \cite{4} proved that every subcubic planar graph with girth at least 23 has a $(2, 2, 2, 2)$-packing coloring. Bre\v{s}ar, Gastineau and Togni \cite{11} proved that every subcubic outerplanar graph has a $(1, 2, 2, 2)$-packing coloring and their result is sharp in the sense that there exist subcubic outerplanar graphs that are not $(1,2,2,3)$-packing colorable. Moreover, Kostochka and Liu \cite{neww} recently showed that every 2-connected subcubic outerplanar graph is $(1,1,2)$-packing colorable and every subcubic outerplanar graph is $(1,1,2,4)$ packing-colorable. On our way to proving Conjecture 1 for a subclass of subcubic graphs which are Halin, and depending on the observation of Gastineau and Togni, it would have been sufficient to prove that such graphs are $(1,1,2,2)$-packing colorable, but we proved a stronger result which is that cubic Halin graphs are $(1,1,2,3)$-packing colorable, and this answers the question posed by Gastineau and Togni \cite{16} for these graphs.\\ Concerning the $(1,2,2,\dots )$-packing coloring of cubic Halin graphs, we prove in Section 3 that cubic Halin graphs are $(1,2,2,2,2,2)$-packing colorable.

%\newpage
\section{$(1,1,2,3)$-packing coloring of cubic Halin graphs.}
 \begin{theorem}\label{thm1}
 %If $G$ is a 
 Every cubic Halin graph is $(1,1,2,3)$-packing colorable.
 \end{theorem} 
\begin{proof}
Let $G=T\cup C$ be a cubic Halin graph. The set of colors that we are going to use is $\{1,1',2,3\}$.\\First, we color $T$ by the coloring $\phi _{T}$ in which we use the colors $\{1,1'\}$, this is possible since $T$ is a tree, hence $2$-colorable. Going back to the graph $G$, by adding the edges of the cycle, conflicts may appear on the vertices of $V(C)$ such that two consecutive vertices on $C$ may be receiving a same color. %So we are going to recolor the vertices of the cycle by benefiting from the rema may remove as much colorings of $V(C)$ as necessary, 
So we are going to recolor the vertices of the cycle by benefiting from the remaining colors.\\ Let $n=|V(C)|$ and $C=a_{1}a_{2}\dots a_{n}$ arranged in an increasing order according to a clockwise direction.\\ We have two main cases to study: \paragraph{Case 1:}There exist two consecutive vertices $a_{i_{0}}$ and $a_{i_{0}+1}$ on $C$ such that $\phi _{T}(a_{i_{0}})=1'$ and $\phi _{T}(a_{i_{0}+1})=1$.\\ Let $u$ and $v$ be two vertices on the cycle $C$, we define $d_{C}(u,v)$, the distance between $u$ and $v$ given by the cycle, as $\min \{l(C_{[u,v]}),l(C_{[v,u]})\}$.
 As a first step, we are going to recolor $V(C)$ by a $(1,1,2,3)$-packing coloring $\phi _{C}$ such that this coloring considers $d_{C}(u,v)$ (we mean by this that when giving two vertices same color $i$, we must make sure that the distance $d_{C}$ between these two vertices is at least $i+1$). Then we will deal with the conflicts that appear due to the distance given by $G$ since we possibly have $d_{C}(u,v)> d_{G}(u,v)$ for some $u,v\in C$.\\ $\bullet$ If $n\equiv 0\bmod{4}$ then we color $V(C)$ by repeating the sequence $i2i3$ where $i\in \{1,1'\}$ depending on $\phi_{T}$.\\ $\bullet$ If $n\equiv 1\bmod{4}$, then we color $V(C)$ starting from $a_{i_{0}+1}$ and repeating the sequence $i2i3$ where $i\in \{1,1'\}$ depending on $\phi _{T}$, this is possible since $\phi _{T}(a_{i_{0}})=1'$.\\ $\bullet$ If $n\equiv 2\bmod{4}$, then we color $V(C_{[a_{i_{0}+1},a_{i_{0}-2}]})$ starting from $a_{i_{0}+1}$ and repeating the sequence $i2i3$ where $i\in \{1,1'\}$ depending on $\phi _{T}$. If $\phi _{T}(a_{i_{0}-1})=1$ then we are done, otherwise we give $a_{i_{0}-1}$ the color $2$. \\ $\bullet$ If $n\equiv 3\bmod{4}$, then we color $V(C)$ starting from $a_{i_{0}+1}$ and repeating the sequence $i2i3$ where $i\in \{1,1'\}$ depending on $\phi _{T}$, this is possible since $\phi _{T}(a_{i_{0}})=1'$.\\ Now we are going to deal with the conflicts that appear due to the distance given by $G$. We have two types of conflicts:\\ \\ Conflict of type 1:
  \begin{figure}[h]
\centering
\includegraphics[width=0.3\linewidth]{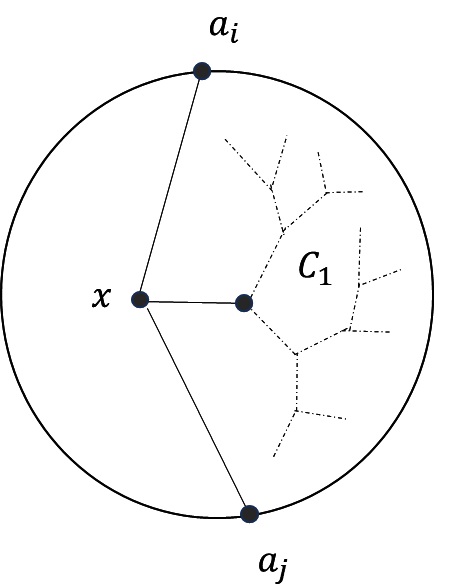}
\caption{ Conflict of type 1: vertices $a_i$ and $a_j$ can not be given both the same color 2 or 3.} %The characteristic tree of $G$ is found in the region $C_{[a_{i},a_{j}]}\cup xa_{i} \cup xa_{j}$.}
\label{fig:conflict1}
\end{figure}There exists $x\in T\backslash C$ and $\{a_{i},a_{j}\}\subset C$ such that $\{xa_{i},xa_{j}\} \subset E(G)$, and both $a_{i}$ and $a_{j}$ receive the color 2 or 3 by the coloring $\phi _{C}$ as illustrated in Fig. 1.\\ We have $d(x)=3$, so we may assume, without loss of generality, that the third neighbor of $x$ is in the region $C_{1}=C_{[a_{i},a_{j}]}\cup xa_{i} \cup xa_{j}$. Since $G$ is connected and planar, then the characteristic tree $T$ of $G$ is contained in the region $C_{1}$. Thus, $a_{j}$ and $a_{i}$ are consecutive on $C$, and this contradicts the fact that the coloring $\phi _{C}$ is a $(1,1,2,3)$-packing coloring of $V(C)$.\\ \\
%\newpage 
Conflict of type 2: 
\begin{figure}[h]
\centering
\includegraphics[width=0.7\linewidth]{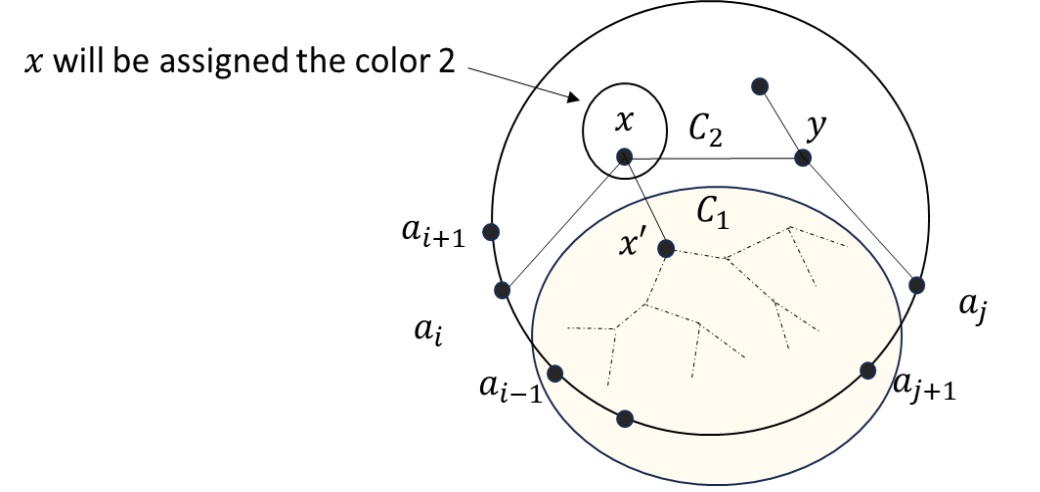}
\caption{ Conflict of type 2: vertices $a_i$ and $a_j$ can not be given both color 3.}
\label{fig:conflict2}
\end{figure}
 There exist $\{x,y\}\subset T\backslash C$, $\{a_{i},a_{j}\}\subset C$ with $j<i$ such that $\{xa_{i},xy,ya_{j}\}\subset E(G)$, and both $a_{i}$ and $a_{j}$ receive the color 3 by the coloring $\phi _{C}$, see Fig. 2.\\ In what follows, we are going to deal with such type of conflict. Since the vertices of $C$ are colored by the packing coloring $\phi _{C}$, and both  $a_{i}$ and $a_{j}$ receive the color $3$ according to $\phi _{C}$, then $d_{C}(a_{j},a_{i})\geq 4$. Thus clearly they are not consecutive on $C$ and so the third neighbor of $x$ and that of $y$ must exist in the two different regions of $C$ which are $C_{1}=C_{[a_{j},a_{i}]}\cup xa_{i} \cup xy \cup ya_{j}$ and $C_{2}=C_{[a_{i},a_{j}]}\cup xa_{i} \cup xy \cup ya_{j}$ since otherwise the characteristic tree will be contained in one of the two regions and $a_{i}$ and $a_{j}$ will be consecutive which is a contradiction. We may assume, without loss of generality, that the third neighbor of $x$, say $x'$, is in the region $ C_{1}$.\\ Now we assign to $x$ the color $2$, and then we recolor $a_{i}$ by the color $\alpha \in\{1,1'\}$, chosen such that $\alpha \neq \phi _{C}(a_{i+1})$.\\ If $\alpha =\phi _{C}(a_{i-1})$ then we switch the colors $1$ and $1'$ of the vertices in the region $C_{1}$ except the vertex $y$.
 % $T_{x'}$, where $T_{x'}$ is the subtree of $T$ with root $x'$ which is included in the region $C_{1}$. 
This switch of the colors is possible since $x$ is now colored $2$.\\ We have $x'\notin C$, since otherwise $C_{[a_{j},a_{i}]}=a_{j}x'a_{i}$ and so $d_{C}(a_{i},a_{j})=2$, a contradiction with the packing coloring $\phi _{C}$. Note that if $x'$ has a neighbor on $C$ then this neighbor must be either $a_{j+1}$ or $a_{i-1}$, since otherwise we get a contradiction with the fact that $G$ is planar and connected.\\ Now we are going to deal with conflicts appearing due to giving $x$ the color $2$.\\ \\ $(i)$ A conflict may appear if $x'$ is adjacent to $a_{j+1}$ and $\phi _{C}(a_{j+1})=2$. According to sequences of colors that we followed, the color $3$ is followed by the color $2$ only in the case when $n\equiv 2\bmod{4}$ in which the pattern of colors is $(i2i3)^*21'$ where this notation $(i2i3)^*$ denotes the repetition of the sequence $i2i3$. So $C_{[a_{j},a_{i}]}$ is colored as follows : $321'(i2i3)^*i2i\alpha $.\\ Now we assign to the vertex $a_{j+1}$ the color $\phi _{T}(a_{j+1})$, assign to the vertex $a_{j+2}$ the color 2, switch the colors $2$ and $3$ in $C_{[a_{j+3},a_{i-1}]}$.\\ Thus $C_{[a_{j},a_{i}]}$ will have the following pattern: $3 \phi _{T}(a_{j+1}) 2 (i3i2)^*i3i\alpha $.\\ \\ $(ii)$ A conflict may appear if $\phi _{C}(a_{i+1})=2$. As before, we notice that this pattern, which is color 3 followed by color 2, only appears in the case when $n\equiv 2\bmod{4}$. Then $C_{[a_{i},a_{j}]}$ is colored as follows : $\alpha 21'(i2i3)^*$. So we assign to $a_{i+1}$ the color $3$.\\ Note that the conflicts $(i)$ and $(ii)$ do not occur at the same time, because this pattern exists once in the coloring of $V(C)$. So the coloring of $a_{i+1}$ in $(ii)$ does not make conflict with the switch of the colors $2$ and $3$ mentioned in $(i)$.\\ \\ $(iii)$ A conflict may appear if $a_{i-1}$ receives the color $2$. This is not possible since according to the sequence of colors used to color $V(C)$, the color $2$ is not followed by $3$ in any of the cases.\\ After dealing with the conflict of type 2, by giving $x$ the color $2$, we want to make sure that if there exist two or more vertices in $T\backslash C$ having the same situation as $x$ and so receiving the color $2$, then they do not make a conflict with each other.\\ Let $S=\{x_{i}, i\in I\}$ be the set of vertices in $T\backslash C$ receiving color $2$ by treating conflicts of type 2. For every $x_{i}\in S$, set $\{a_{i_{1}},a_{i_{2}}\}\subset C$  with $i_{1}<i_{2}$, and $y_{i}\in T\backslash C$ to be the vertices such that $x_{i}$ and $y_{i}$ are adjacent in which one of them is adjacent to $a_{i_{1}}$ and the other is adjacent to $a_{i_{2}}$ where $\{a_{i_{1}},a_{i_{2}}\}$ receive the color $3$. Also for every $x_{i}$, let $C^{i}_{1}$ and $C^{i}_{2}$ be the two regions of the cycle defined as $C_{1}$ and $C_{2}$ are defined previously. Let $x_{i}$ and $x_{j}$ be two vertices in $S$, the new color assigned to these vertices is the color $2$.\\ Note that:
  \begin{equation} N(x_{i})\subseteq C^{i}_{1}\, and\,  N(x_{j})\subseteq C^{j}_{1} \tag{*} \end{equation} (since as we mentioned before, the vertices in $S$, that are assigned the new color 2, are chosen according to this location of their neighbors).\\ In what follows, we are going to study three cases.\\ Subcase 1: $[a_{i_{1}},a_{i_{2}}]\cap [a_{j_{1}},a_{j_{2}}]=\emptyset $. By using (*), we get that neither $x_{i}y_{j}$ nor $x_{j}y_{i}$ belongs to $E(G)$. So $d_{G}(x_{i},x_{j})\ge 3$ and there exists no conflict.\\ Subcase 2: $a_{i_{1}}=a_{j_{2}}$. By using (*), we conclude that $C=C_{[a_{j_{1}},a_{i_{2}}]}\cup a_{i_{2}}a_{j_{1}}$, and so $a_{i_{2}}$ and $a_{j_{1}}$ are consecutive on $C$ and both receiving a color 3 according to the coloring $\phi _{C}$, which is a contradiction.\\ Subcase 3: $[a_{j_{1}},a_{j_{2}}]\subset [a_{i_{1}},a_{i_{2}}]$. Conflict may appear only if $x_{i}$ and $y_{j}$ are adjacent, but if such an edge exists then we get that $a_{i_{1}}$ and $a_{j_{1}}$ are consecutive on $C$, a contradiction with the coloring $\phi _{C}$.\paragraph{Case 2:} All the vertices of $C$ receive the same color according to $\phi_{T}$. \\We may suppose, without loss of generality, that all the vertices of $C$ receive the color 1.\\ First note that if $n\equiv 0 \bmod{4}$, then $\chi _{p}(C)=3$ and so we can color $C$ by the colors $\{1,2,3\}$ by repeating the sequence $1213$.\\ In what follows, we suppose that $n\bmod{4}\in\{1,2,3\}$.\\ Our aim is to choose a vertex from $T\backslash C$ to give it the color $2$, and give its neighbor on $C$ the color $1'$, and then deal with the problem similarly to the first case.\\Since all the vertices of $C$ receive the same color according to $\phi_{T}$, then there exists no path $P$ of odd length joining two vertices $a_{i}$ and $a_{j}$ on $C$ such that $P\cap C=\{a_{i},a_{j}\}$.\\ Now we suppose that there exists a vertex $w\in T\backslash C$ such that $w$ has exactly one neighbor on $C$. 
\begin{figure}[h]
\centering
\includegraphics[width=0.6\linewidth]{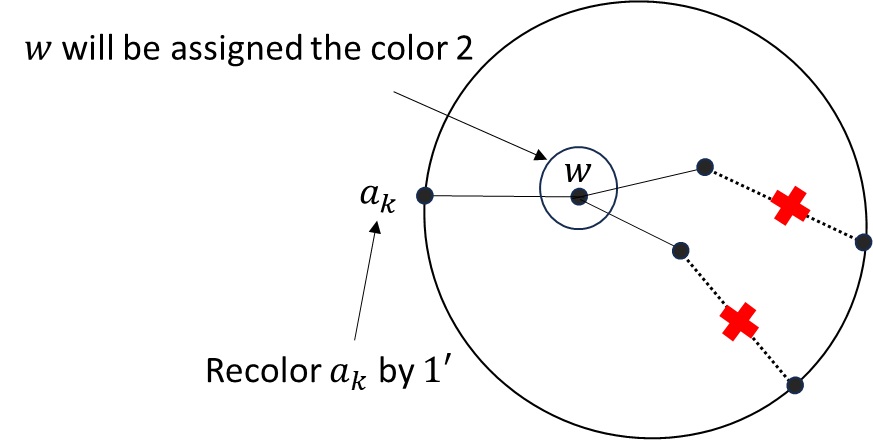}
\caption{$w$ has one neighbor on $C$ and there is no odd path between two vertices of $C$ that passes through $w$.} %The characteristic tree of $G$ is found in the region $C_{[a_{i},a_{j}]}\cup xa_{i} \cup xa_{j}$.}
\label{fig:case2}
\end{figure}  
  
  Assign to $w $ the color $2$ and give its neighbor on $	C$, say $a_{k}$, the color $1'$.\\ If $n\equiv 1\bmod{4}$, then we color $V(C\backslash \{a_{k}\})$ starting from $a_{k+1}$ by repeating the sequence $1213$. Vertex $w$ has no neighbors on $C$ other than $a_{k}$, so let $w_{1}$ and $w_{2}$ be the neighbors of $w$ in $T\backslash C$. A conflict may appear if $w_{i}$ for some $i\in\{1,2\}$ is adjacent to a vertex $u$ on $C$ receiving the color $2$, but $w_{i}$ has no neighbors on $C$ since otherwise $a_{k}ww_{i}u$ is an odd length path, see Fig. 3.\\ If $n=2 \bmod 4$. Color $V(C\backslash \{a_{k}\})$ starting from $a_{k+1}$ by repeating the sequence $1213$. This coloring assigns to $a_{k-1}$ the color $1$, so it has no conflict with $w$. Also, as in the previous case, the neighbors of $w$ make no conflicts.\\ If $n=3 \bmod 4$, then color $V(C\backslash \{a_{k},a_{k-1},a_{k-2}\})$ starting from $a_{k+1}$ by repeating the sequence $1213$, and assign to $a_{k-1}$ the color $1$ and to $a_{k-2}$ the color $2$.\\ \\  Let $M$ be the set of vertices in $T\backslash C$ that are adjacent to $C$, and suppose it is of cardinality $m$. Now assume there exists no vertex in $T\backslash C$ which has exactly one neighbor on $C$. So, every vertex in $M$ has two neighbors on $C$. It is clear, by the definition of the Halin graph, that any vertex on $C$ is adjacent to one vertex in $M$. Therefore, we get that $n=2m$, and so $n$ is even and we are in the case $n=2 \bmod 4$. %Now assume there exists no vertex in $T\backslash C$ which has exactly one neighbor on $C$. So all the vertices in $T\backslash C$ that are adjacent to $C$ have two neighbors on $C$. It is clear that any vertex on $C$ is adjacent to one vertex in $T\backslash C$. Therefore, we get that $n$ is even and we are in the case $n=2$ mod $(4)$.
  Let $w\in T\backslash C $, and let $a_{i}$ and $a_{j}$ be the neighbors of $w$ on $C$ with $i<j$. Since $G$ is planar and connected, we conclude that $a_{i}$ and $a_{j}$ are consecutive. We assign to $w$ the color 2 and to $a_{j}$ the color $1'$ and color $ V(C\backslash \{a_{j}\})$ starting from $a_{j+1}$ till reaching $a_{i}$ by repeating the sequence $1213$.\\ \\
  Finally, in Case 2, conflicts of type 1 are treated similarly as in Case 1. And since there exists no path $P$ of odd length joining two vertices $a_{i}$ and $a_{j}$ on $C$ such that $P \cap C = \{ai, aj \}$, so we have no conflicts of type 2. This completes the proof.  
 \end{proof}
 
 \subsection{Sharpness of the proved result}
 First, it should be noted that without the condition that the graph is cubic, it is easy to find examples proving that a Halin graph is not $(1,1,2,3)$-packing colorable in general.\\
 Moreover, the result of Theorem 1 is sharp in the sense that there exists a cubic Halin graph that is not $(1,1,3,3)$-packing colorable.\\ \begin{claim} The graph $G_{1}$ shown in Figure 4, which is clearly cubic Halin, is not $(1,1,3,3)$-packing colorable.\end{claim} 
 \begin{proof}
     
Suppose that $G_{1}$ is $(1,1,3,3)$-packing colorable and let $\{1_{a},1_{b},3_{a},3_{b}\}$ be the set of colors used to color it.
The graph $G_{1}$ has three disjoint triangles $T_{1},T_{2}$ and $T_{3}$. So, in each triangle, there exists a vertex that must receive a color from the set $\{3_{a},3_{b}\}$. But the diameter of the graph is $3$, which is a contradiction.   \end{proof}
% \begin{proof} Suppose that $G_{1}$ is $(1,1,3,3)$-packing colorable and let $\{1_{a},1_{b},3_{a},3_{b}\}$ be the set of colors that we will use. For every subgraph $T_{i}$ of $G$, with $i\in \{1,2,3\}$, there exists a vertex $v_{i} \in T_{i}$ such that the coloring of $v_{i}$ must be chosen from the set $\{3_{a},3_{b}\}$. So we get that two vertices from $\{v_{1},v_{2},v_{3}\}$ must receive the same color, without loss of generality, say $3_{a}$. But the diameter of the graph is $3$, which is a contradiction.\end{proof} 
Since $G_1$ is not $(1,1,3,3)$-packing colorable, it is not $(1,2,3,4)$-packing colorable either. Therefore, cubic Halin graphs are neither $(1,1,3,3)$-packing colorable nor 4-packing colorable in general.
 \begin{figure}[h]
\centering
\includegraphics[width=0.3\linewidth]{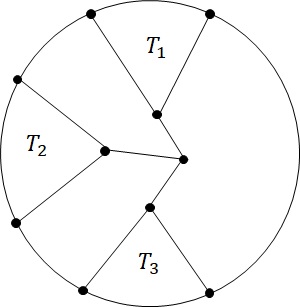}
\caption{ A non $(1,1,3,3)$-packing colorable cubic Halin graph $G_{1}$}
\label{fig: example}
\end{figure}
\section{$(1,2,2,2,2,2)$-packing coloring of cubic Halin graphs}
Observe first that any tree $T$ of order at least $4$ such that all its vertices are either of degree $1$ or $3$, contains a vertex that is adjacent to at least two leaves. This is easy to be proved, so we will omit its proof.

Moreover, as a conclusion of the previous remark and by using the induction, we get that such type of tree has an even number of vertices.

\begin{lemma} \label{ll} Let $T$ be a tree of order at least $4$, such that all its vertices are either of degree $1$ or $3$. Then $T$ admits a $(1,2,2,2)$-packing coloring in which all the leaves are given the color 1.\end{lemma} 

\begin{proof} We will proceed the proof by induction on $v(T)=|V(T)|$.\\ For $v(T)=4$, it is clearly true. Let us prove it for $v(T)=m$, knowing that it is true up to $m-2$. Let $T$ be a tree of order $m$ satisfying the above conditions. Let $x$ and $y$ be two leaves that are adjacent to the same vertex $z$. Let $T'=T\backslash \{x,y\}$. By applying the induction hypothesis on $T'$, we get that $T'$ admits a $(1,2,2,2)$-packing coloring such that all its leaves are colored by 1. Going back to $T$, by adding the vertices $x$ and $y$, we have all the leaves of $T$ are colored by 1 except $x$ and $y$. The color assigned to $z$ is 1 because it is a leaf in $T'$. Now we are going to give both vertices $x$ and $y$ the color 1, and then recolor $z$. Note that the set of colors that we are using is $\{1,2_{a},2_{b},2_{c}\}$. Let $z'$ be the third neighbor of $z$ and $N^{2}$ be the set of neighbors of $z'$ other than $z$. If there exists a color $i\in \{2_{a},2_{b},2_{c}\}$ such that $i$ is not given to any of the vertices of $N^{2}\cup \{z'\}$, then we assign to $z$ the color $i$. Otherwise, we give $z'$ the color 1 and assign to $z$ the color that was previously given to $z'$ from the set $\{2_{a},2_{b},2_{c}\}$. Now $T$ admits a $(1,2,2,2)$-packing coloring in which all its leaves are given the color 1, and this completes the proof.  \end{proof} 

\begin{theorem} If $G$ is a cubic Halin graph, then $G$ is $(1,2,2,2,2,2)$-packing colorable. \end{theorem}

\begin{proof} Let $G=T\cup C$ be a cubic Halin graph and $\{1,2_{a},2_{b},2_{c},2_{d},2_{e}\}$ be the set of colors that we are going to color $G$ with. First, by using Lemma \ref{ll}, we use the set of colors $\{1,2_{a},2_{b},2_{c}\}$ to color the vertices by a $(1,2,2,2)$-packing coloring such that all the vertices of the cycle $C$, which are the leaves of the tree $T$, are colored by $1$. Then, except if $|C|=5$, we are going to recolor the vertices of $C$ by combining the sequences $12_{d}2_{e}$ and $12_{d}12_{e}$ depending on the order of the cycle. Since any number $n\geq 4$ different from 5 is a linear combination of $3$ and $4$, then we can combine the sequences. Similarly to the proof of Theorem~\ref{thm1}, it can be shown that conflicts of type 1 are not possible hence the obtained coloring is a $(1,2,2,2,2,2)$-packing coloring. If $|C|=5$, then $G$ is unique and has order 8, see Fig. \ref{fig: exampl}, and it can be easily $(1,2,2,2,2,2)$-packing colored. This completes the proof. \end{proof}
\begin{figure}[h]
\centering
\includegraphics[width=0.3\linewidth]{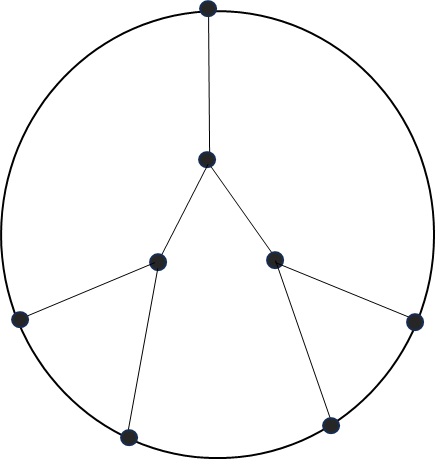}
\caption{ The unique cubic Halin graph $G$ having five vertices on its cycle.}
\label{fig: exampl}
\end{figure}

 Notice that the unique cubic Halin graph on 6 vertices is not $(1,2,2,2)$-packing colorable, however we are not able to find a cubic Halin graph that is not $(1,2,2,2,2)$-packing colorable, hence the result of the previous theorem is maybe not sharp.
\section{Open questions}
Our interests are directed toward answering the following questions, either by proving their correctness or by giving a counter-example:

Is every cubic Halin graph $(1,1,2,4)$-packing colorable? $(1,1,2,5)$-packing colorable? 

Is every cubic Halin graph $(1,2,2,2,2)$-packing colorable?

\paragraph{Acknowledgment.} The authors would like to acknowledge the National Council for Scientific Research of Lebanon (CNRS-L) and the Agence Universitaire de la Francophonie in cooperation with Lebanese University for granting a doctoral fellowship to Batoul Tarhini.

\end{document}